\documentclass[11pt,twoside, letter]{article}
\usepackage{amsthm, amsmath, amssymb, amsfonts, graphicx, epsfig}
\usepackage{accents}

\usepackage{bbm}
\usepackage{mathrsfs}

\usepackage{cancel}

\usepackage{epstopdf}

\usepackage{graphicx}
\usepackage[font=sl,labelfont=bf]{caption}
\usepackage{subcaption}

\usepackage{float}
\restylefloat{table}

\usepackage{enumerate}

\usepackage[usenames,dvipsnames]{color}

\usepackage[colorlinks=true, pdfstartview=FitV, linkcolor=blue,
            citecolor=blue, urlcolor=blue]{hyperref}
\usepackage[usenames]{color}

\usepackage{url}
\makeatletter\def\url@leostyle{%
 \@ifundefined{selectfont}{\def\UrlFont{\sf}}{\def\UrlFont{\scriptsize\ttfamily}}} \makeatother

\usepackage[framemethod=default]{mdframed}

\usepackage{array}

\usepackage{mathtools}
\mathtoolsset{showonlyrefs=true}

\usepackage[margin=1.0in, letterpaper]{geometry}

\newtheorem{theorem}{Theorem}

\theoremstyle{definition}

\theoremstyle{remark}
\newtheorem{remark}[theorem]{Remark}

\numberwithin{equation}{section}
\numberwithin{theorem}{section}

\definecolor{Red}{rgb}{0.9,0,0.0}
\definecolor{Blue}{rgb}{0,0.0,1.0}

\def\cB{\mathcal{B}}

\def\cD{\mathcal{D}}

\def\cK{\mathcal{K}}

\def\cM{\mathcal{M}}
\def\cN{\mathcal{N}}

\def\bE{\mathbb{E}}

\def\bN{\mathbb{N}}

\def\bP{\mathbb{P}}

\def\bR{\mathbb{R}}

\def\sF{\mathscr{F}}

\def\mV{\mathsf{V}}

\newcommand{\wh}{\widehat}

\newcommand{\pd}[1]{\partial_{#1}}      
\newcommand{\set}[1]{\{#1\}}            
\renewcommand{\mid}{\;|\;}              
\newcommand{\norm}[1]{ \| #1 \| }       

\DeclareMathOperator{\dif}{d \!}        

\DeclareMathOperator{\Var}{Var}          

\title{Statistical Inference for SPDEs: an overview}

\makeatletter
\def\and{%
  \end{tabular}%
  \begin{tabular}[t]{c}}%
\def\@fnsymbol#1{\ensuremath{\ifcase#1\or a\or b\or c\or
   d\or e\or f\or g\or h\or i\else\@ctrerr\fi}}
\makeatother

\author{
    Igor Cialenco \\[0.3ex]
        {\footnotesize Department of Applied Mathematics, Illinois Institute of Technology} \\
        {\footnotesize 10 W 32nd Str, Building E1, Room 208, Chicago, IL 60616, USA}\\
          \url{cialenco@iit.edu},   \url{http://math.iit.edu/\~igor}
         }

\date{ {\small   %
First Circulated: December 1, 2017
}}

\begin{document}

\maketitle

\vspace{-2em}

\smallskip

{\footnotesize
\begin{tabular}{l@{} p{350pt}}
  \hline \\[-.2em]
  \textsc{Abstract}: \ & The aim of this work is to give an overview of the recent developments in the area of statistical inference for parabolic stochastic partial  differential equations. Significant part of the paper is devoted to the spectral approach, which is the most studied sampling scheme under which the observations are done in the Fourier space over some finite time interval. We also discuss into details the practically important case of discrete sampling of the solution. Other relevant methodologies and some open problems are briefly discussed over the course of the manuscript.   \\[0.5em]
\textsc{Keywords:} \ &  parabolic SPDE, stochastic evolution equations, statistical inference for SPDEs, identification problems for SPDEs \\
\textsc{MSC2010:} \ & 60H15, 65L09, 62M99   \\[1em]
  \hline
\end{tabular}
}



\section{Introduction}\label{sec:Intro}

The general analytical theory for (linear and nonlinear) Stochastic Partial Differential Equations (SPDEs)  went through major advances during the past few decades and became a mature mathematical field. From practical point of view, SPDEs are used to describe and model the evolution of dynamical systems in the presence of persistent spatial-temporal uncertainties, and are key ingredients in modeling various phenomena from fluid mechanics, oceanography, temperature anomalies, finance, economics, biological and ecological systems, and many other applied disciplines. We refer to the classical monographs \cite{RozovskiiBook,DaPrato}, and also to the textbooks \cite{ChowBook,HairerBook2009,LototskyRozovsky2017Book}, for an in depth discussion of the theory of SPDEs and their various applications. While the general form of a particular stochastic evolution equation is commonly derived from the fundamental properties of the underlying processes under study, frequently the parameters arising in the formulation need to be specified or determined on the basis of some empirical observations. Moreover, even if the parameters are known as part of the specification of the underlying model, the observer may need to test how well the empirical data fits the considered model. These inverse type problems, naturally arising in practical applications, fall into realm of the well developed field of statistical inference for stochastic processes. Nevertheless, due to the infinitesimal nature of SPDEs, some of the statistical models arising from SPDEs are fundamentally different from their counterpart in (finite dimensional) stochastic ordinary differential equations (SODEs). Albeit, there exits  a good number of papers devoted to parameter estimation problems for SPDEs, it would be fair to say that the field of statistical inference for SPDEs is still in its developing stage with many fundamental problems still open.

The aim of this paper is to give an overview of the main methodologies developed in the area of statistical inference for SPDEs. Due to the page limitation, the detailed proofs will be omitted, and some results will be mentioned only briefly. We will also take as known most of the notions and fundamental results from infinite dimensional stochastic analysis. The paper is divided in (sub)sections, each of them being devoted to a general method or problem. All the relevant literature and the obtained results will be discussed within the corresponding subsection.

\smallskip

\noindent\textbf{The equation.} Throughout, we will assume that $(\Omega,\sF,\set{\sF_t}_{t\geq 0},\bP)$ is a stochastic basis satisfying the usual assumptions, and let $H$ be a separable Hilbert space endowed with the inner product $(\,\cdot\,,\,\cdot\,)$ and the corresponding norm $\norm{\cdot}$. The main object of interest is the following evolution equation
\begin{equation}\label{eq:mainAbstract}
\dif u(t) + \left(\theta A +  B\right)u(t) \dif t  = (M u(t) + \sigma) \dif W^Q(t),
\end{equation}
with the initial condition $u(0)=u_{0}\in H$, and where $A$ is a linear, positive defined, self-adjoint operator in $H$, $\cB$ is a linear or nonlinear operator in $H$, $\cM$ is an operator acting in a suitable Hilbert space,  $W^Q$ is a $Q$-cylindrical Brownian motion in $H$, $\sigma$ and $\theta$ are some positive constants. We will refer to $W^Q$ with $Q$ being the identity operator as space-time white noise, and simple write $W$. Although most of the results presented below hold true for general evolution equations under fair general assumptions, practically speaking, usually $A$ and $M$ are some (pseudo)differential operators, hence SPDEs. Since most of the literature on statistical inference for SPDEs is devoted to parabolic equations, we will assume that \eqref{eq:mainAbstract} is parabolic. Moreover, we will make the standing assumption that equation \eqref{eq:mainAbstract} has a unique solution in $H$ (or more precisely in the appropriate triple of Hilbert spaces). The solution is either weak, mild or strong in the PDE sense, and strong in the  probability sense. Of course, for each particular class of considered equations, one has to show that indeed the existence and uniqueness of the solution holds true; most of the examples considered here can be dealt within the abstract framework presented in \cite[Section~4.4]{LototskyRozovsky2017Book} and/or \cite{ChowBook,HairerBook2009,RozovskiiBook,DaPrato}.

\smallskip \noindent
\textbf{The statistical problem.}
We will be mainly interested in statistical inference problems related to the model parameters $\theta$ and/or $\sigma$, assuming that all other quantities are known to the observer. We will call $\theta$ the drift or viscosity coefficient, and $\sigma$ the volatility. If the only parameter of interest is $\theta$, then all the obtained results can be easily adjusted to the case when $\sigma$ is an adapted vector-valued function.

\smallskip
\noindent\textbf{Statistical inference for SPDEs vs SODEs.}
Let us consider a finite dimensional diffusion process of the form $\dif X(t)= \mu a(X)\dif t + \eta  b(X(t))\dif w(t)$, with $\mu$ and $\eta$ being the parameters of interest, and $a,b$ some suitable, real valued functions. Let us assume that the solution is observed continuously  over some finite interval of time. It is well known that estimating the volatility $\eta$ is a singular problem, namely the measures $\set{\bP_{\eta,T}, \eta>0}$, on the space of trajectories $C([0,T];\bR)$ generated by the solution $X$ are singular to each other. This indicates that $\sigma$ can be found exactly. For example,
using the `quadratic variation argument', we get $\eta = (\langle X \rangle_T/\int_0^TX^2(t)\dif t)^{\frac12}$. Of course, if the data is observed discretely, then we can only estimate $\eta$, for example by approximating the above formula, or by using some statistical methods.  Consistency of such estimators can be achieved by decreasing the frequency of the time step while keeping the time horizon $T$ fixed, or increasing the observation time while step size is fixed, or both.

On the other hand, estimating $\mu$ usually is a regular problem, in the sense that the measures $\set{\bP_{\mu,T}, \mu\in\bR}$, generated by the solution $X$ the SODE with drift $\mu$, are absolutely continuous to each other, and one can employ the statistical inference theory for stochastic processes to find some estimators for $\mu$, for example the maximum likelihood estimator (MLE).  In this case, the only way to find $\mu$ is either increase the time (large time asymptotics $T\to\infty$), or decrease the noise (small noise asymptotics $\eta\to0$). Discrete time observations will not change the nature of the problem.
Similar conclusions hold true for finite dimensional stochastic differential equations (SODEs), and there are known necessary and sufficient conditions for measures $\bP_{\mu,T}$ to be absolutely continuous to each other. We refer to classical monographs \cite{LiptserShiryayev,LiptserShiryayevBook1978,KutoyantsBook2004} for details on statistical inference for SODEs, and to \cite{Kozlov1977,Kozlov1978,KoskiLoges1985,Loges1984,MikuleviciusRozovskii1994,MikuleviciusRozovskii2002} for extension of Girsanov type theorems and absolute continue of measures  to infinite dimensional spaces.

As far as estimating volatility $\sigma$ for SPDEs, this is essentially the same as for SODEs. The key difference in parameter identification problems for SPDEs vs SODEs comes in estimating the drift $\theta$ - usually, if $A$ is the leading order operator, the measures $\set{\bP_{\theta,T},\ \theta>0}$, generated by the solution $u, \ t\in[0,T]$, are mutually singular. This indicates that the drift $\theta$ can be found exactly, if one path of the solution is observed as an element of the Hilbert space $H$ and continuously over a finite time interval. As with any singular problem, there is no universal method to study it, and each problem has to be analyzed separately. In many situations, a singular model can be approximated by regular models, and for a large class of SPDEs this can be achieved by applying MLE to a projection of the solution on an appropriate finite dimensional subspace of $H$; see Section~\ref{sec:spectral}. Alternatively, for some particular classes of SPDEs


\smallskip
\noindent\textbf{The observations.} As with any statistical inference problem, the first question to answer is what data is available to the observer.

From practical point of view, usually, the observer will take measurements of the solution $u(t,x)$ at some \textit{discrete time} points $t_i$ and/or some \textit{discrete spatial} points $x_j$, over one path/realization $\omega\in\Omega$. While practical, only few classes of linear SPDEs have been fully investigated under this sampling scheme; see Section~\ref{sec:discrete}.

Analogous to inverse problems for diffusion processes, it is reasonable to take  the continuous time and/or space observation approach for SPDEs too. Most of the literature on this subject for SPDEs is devoted to the so called \textit{spectral approach}, when the experimenter observes continuously in time one trajectory of  a finite number $N$ of the Fourier modes over some finite interval of time. We discuss this sampling scheme in Section~\ref{sec:spectral}. Note that in this case, estimating volatility $\sigma$ is a trivial task, and $\sigma$ can be found exactly by using quadratic variation argument, similar to finite dimensional diffusions. Thus, in the spectral approach, we focus only on estimating $\theta$.

In the context of the proposed observation schemes, there are several asymptotic regimes that one can study. Assuming spectral approach, fix $N$ and study the long time asymptotics $T\to\infty$. This usually reduces the problem to a finite dimensional system of SODEs that has been thoroughly investigated in the existing literature; cf. \cite{KutoyantsBook2004} and references therein. We will omit this analysis here, and for more details specific to SPDEs, we refer the reader to \cite{LototskyRozovsky2017Book,CXu2015} for the spectral approach, to \cite{KoskiLoges1985,Loges1984} for non-spectral approach, and to \cite{GoldysMaslowski2002} for controlled SPDEs.
As already mentioned, with the idea of approximating a singular model by regular models, in the spectral approach, we will take the large number of Fourier modes asymptotics, $N\to\infty$, which will be one of the main focuses of this paper; see Section~\ref{sec:spectral}. To some extent, this regime corresponds to the fine spacial resolution - at least from approximation point of view, the more sampling  points available in the physical domain, the more and better Fourier modes can be approximated, and knowing the solution for all spacial points is equivalent to knowing all Fourier coefficients of the solution. Another natural asymptotics to consider in spectral approach is long time and large space resolution $T,N\to\infty$, which was not yet study in the existing literature.

Given the singular nature of the problem, in the discrete sampling regime it makes sense to study first the asymptotics in large number of spacial points (fine spacial resolution) and/or large number of time points (fine time-step size), while keeping the time horizon $T$ fixed. We study these in Section~\ref{sec:discrete}. Alternatively, one can also study large time asymptotics, or a combination of all the above, which are still open problems.

Finally, a classical asymptotic regime for estimating $\theta$ is small noise asymptotics, i.e. the scaler coefficient in front of the noise goes to zero.
The literature on this regime is limited \cite{IbragimovKhasminskii1998,IbragimovKhasminskii1999,IbragimovKhasminskii2000,Huebner1999,PrakasaRao2003}. While practically speaking the observer rarely deals with vanishing noises, these are interesting problems to investigate.

Note that all the above sampling scheme assumed that only one realization or path of the solution is observed. Sampling multiple paths is another possibility, and besides elevating the methods developed for one path, one can use energy type estimates of  the solution to build statistical estimates for the parameters of interests.

\smallskip
\noindent\textbf{Classes of SPDEs.} It is well known that the structure of the noise of an SPDE, in our case the right hand side of \eqref{eq:mainAbstract}, may affect significantly the properties of the solution, and more importantly, the technical difficulties in proving certain results. The two main classes of noise structures are: \textit{additive noise}, when $M=0$, i.e. the solution is not part of the noise, and \textit{multiplicative noise} otherwise. Usually, dealing with multiplicative noise is more challenging, and statistical inference for SPDEs is not an exception in this regards.
There are only few results on parameter estimation problems for SPDEs driven by multiplicative noise \cite{PospivsilTribe2007,CialencoLototsky2009,Cialenco2010,CialencoHuang2017}, primarily due to the fact that the main methods of proofs exploit the diagonalizable structure of the additive noise.

The presence of the nonlinear operator $B$ in \eqref{eq:mainAbstract} makes the SPDE to be nonlinear, and as one may expect, studying such equations usually involves technics specific to the considered  equation.  In Section~\ref{sec:spectral} we will present a general method how to construct and study estimators for $\theta$ in the presence of nonlinear term and driven by an additive noise.

\smallskip\noindent
\textbf{Examples.} We will present several examples of SPDEs that we will consider over the course of this work. Again, we will omit the discussion on the existence and uniqueness of the solutions, and we will only mention that all these equations are well defined and do admit a unique solution in proper spaces under some technical assumptions that do agree with the assumptions encounter by the considered statistical problems and the proposed methods.

\noindent I) \textit{Stochastic heat equation driven by additive noise.} Consider the following evolution equation
\begin{equation}\label{eq:heatAdditive}
  \dif u(t,x) - \theta \Delta u(t,x)\dif t = \sigma \dif W^Q(t,x), \qquad t>0,
\end{equation}
with zero initial condition $u(0,x)=0$, and where $x\in G\subset\bR^d$, $\Delta u = \sum_{k=1}^{d}u_{x_kx_k}$ denotes the Laplace operator, $\theta,\sigma\in\bR_+$. As far as $G$, we will consider: I.a) the case of bounded domain, assuming that $G$ is a smooth, bounded domain in $\bR^d$, and we endow \eqref{eq:heatAdditive} with zero boundary condition; I.b) the case of full space, by taking $G=\bR^d$.

\smallskip\noindent II)
$\dif u(t,x) - (\Delta u(t,x) + \theta u(t,x))\dif t  = \sigma \dif W^Q(t,x)$, $t\in (0,T], \ x\in(0,\pi)$, with zero initial and boundary conditions, and where $\theta\in\bR$ and $\sigma\in\bR_+$.

\smallskip\noindent III) \textit{2D Navier--Stokes Equations forced with additive noise.} Consider the equation that describes the flow of a viscous,
incompressible fluid, known as Navier-Stokes equations, of the form,
\begin{align}\label{eq:NVS}
  \dif u -\theta \Delta u\dif t + (u\cdot\nabla u) u\dif t + \nabla P \dif t & = \sigma \dif W^Q(t), \\
  \nabla \cdot U &=0, \\
  u(0) & = u_0,
\end{align}
where $u$ represent the velocity field and $P$ the pressure.  We consider either zero boundary conditions on some $G\in\bR^2$, or periodic boundary conditions.

\smallskip\noindent IV) \textit{Stochastic heat equation, simple multiplicative noise}. A variation of Example~I,
\begin{equation}\label{eq:heatSimpleM}
  \dif u(t,x) - \theta \Delta u(t,x)\dif t = \sigma u(t,x)\dif w(t), \qquad t>0,
\end{equation}
where $x\in[0,\pi]$, $u(0)=u_0$, and $w$ is a  one dimensional standard Brownian motion.

\smallskip\noindent V) \textit{Stochastic heat equation driven by space-time multiplicative noise}. The multiplicative noise counterpart of Example~I,
\begin{equation}\label{eq:heatMultiple}
  \dif u(t,x) - \theta \Delta u(t,x)\dif t = \sigma u(t,x)\dif W^Q(t,x), \qquad t>0,
\end{equation}
with initial data $u(0)=u_0$, and either bounded domain $x\in G\subset{\bR^d}$ with Dirichlet boundary conditions, or the whole space $x\in\bR^d$.

\section{Spectral approach}\label{sec:spectral}
Most of the results on statistical inference for SPDEs are obtained within the so called spectral approach. The first key assumption is that the equation \eqref{eq:mainAbstract} is `diagonalizable'. That is, we assume that the operators $A$ and $Q$ have pure point spectrum, and a common system of eigenfunctions $\set{h_k}_{k\in\bN}$ that forms a complete orthonormal system in $H$. We will not make any assumptions on $B$ yet.
We denote by $\nu_k$ and $q_k^2$ the eigenvalues of $A$ and $Q$, respectively, corresponding to the eigenfunction $h_k$, for $k\in\bN$. Using this basis we define $H^{N} := \mbox{span}\set{h_{k}: k=1,\ldots,N}$. We denote by $P^{N}$ the projection operator from $H$ onto $H^{N}$, and we will write $u^{N} = P^{N}u = \sum_{k=1}^N u_k h_k$, where $u_k=(u,h_k), \ k\in\bN$, are the Fourier coefficients (or modes) of the solution $u$ with respect to $\set{h_k}_{k\in\bN}$. Note that in this case, $W^Q$ can be formally written as
\[
W^Q(t) = \sum_{k=1}^{\infty} q_k h_k w_k(t),
\]
where $w_k$ are independent standard Brownian motions. Also note that since $ H^{N} \cong \bR^{N}$, we can view $u^N$ as an element of  $\bR^N$ for any fixed $t\in[0,T]$.

Second key postulate in spectral approach is to assume that we observe continuously in time one path of the Fourier modes $u_k(t), k=1,\ldots,N$ over some finite interval of time $[0,T]$.

We are interested in estimating the parameter $\theta\in\Theta\subset\bR_+$, and we will assume that the positive constant $\sigma$ is known.

The class of equations to which the spectral approach can be applied is large and practically important. Usually, we are dealing with pseudo-differential operators on a bounded domain, for which there exist easy to check conditions that guarantee that the spectrum is pure discrete and the set of eigenfunctions form a basis. Moreover, the asymptotic behavior of the eigenvalues $\nu_k$ of $A$ is tightly related to the order $A$ (as a differential operator) and the dimension of the space $\bR^d$, namely\footnote{As usual, for two sequences of positive numbers $\{a_{n}\}_{n\in\bN}$ and $\{b_{n}\}_{n\in\bN}$, we will write $a_{n}\sim b_{n}$ if $\lim_{n\rightarrow\infty}a_{n}/b_{n}=1$, and will write $a_{n}\asymp b_{n}$ if there exist universal constants $K_{2}>K_{1}>0$, such that $K_{1}b_{n}\leq a_{n}\leq K_{2}b_{n}$ for $n\in\mathbb{N}$ large enough.}
$\nu_k \asymp k^{\textrm{ord(A)}/d}$. Thus, if $A = (-\Delta)^\beta$ is the fractional Laplace operator on some $G\subset\bR^d$, then $\nu_k\asymp k^{2\beta/d}$. For this analysis, the particular form of the eigenfunctions is not important. Note that, the differential operator, and hence SPDE, defined on the whole space will not have a pure point spectrum, and the spectral approach in principle can not be applied to these equations.

\subsection{Modified MLEs: general methodology}\label{sec:mMLE}
We will consider equations driven by an additive noise, i.e. $M=0$.

We apply the projection operator $P^N$ to \eqref{eq:mainAbstract}, and obtain
\begin{align}\label{eq:SPDEabsProj}
d u^N + (\theta A u^N(t) + \Psi^N )dt = \sigma dW^{Q,N}(t),
\quad u^N(0) = P^N u_{0},
\end{align}
where $\Psi^N = P^{N}B(u)$, and $W^{Q,N} = P^NW^Q$.

Note that \eqref{eq:SPDEabsProj} generally speaking is not an equation in $u^N$, unless $P^N$ commutes with $B$. Nevertheless, let us pretend that $\Psi^N$ is known, and denote by $\bP_\theta^{T,N}$ the probability measure on $C([0,T]; H^N)$ generated by the solution $u^N$ of \eqref{eq:SPDEabsProj}.
Under the \textit{unjustified} assumption that the family of measures $\{\bP_{\theta}^{T,N}(\cdot)\}_{\theta\in\Theta}$ are mutually absolutely continuous, we now fix a reference (true) parameter $\theta_{0}$ and formally apply the Girsanov theorem (see e.g. \cite[Section~7.6.4]{LiptserShiryayev}) to \eqref{eq:SPDEabsProj}, and we obtain the Radon-Nykodym derivative or Likelihood Ratio
\begin{align}
\frac{\dif \bP_{\theta}^{N}}{ \dif\bP_{\theta_{0}}^{N}} (u^{N}) =
\exp \Biggl(&
- (\theta - \theta_{0}) \sigma^{-2}\int_{0}^{T} (Q^{-1} A u^{N},\dif u^{N}) -
\frac{1}{2}  (\theta^{2} + \theta_{0}^{2})\sigma^{-2} \int_{0}^{T}   \norm{Q^{-1/2}Au^{N}}^2 \dif t
 \notag\\
 &- (\theta - \theta_{0}) \sigma^{-2}\int_{0}^{T} (Q^{-1}A u^{N},\Psi^{N}) \dif t
\Biggr),
\label{eq:likelyhoodRatio}
\end{align}
where with slight abuse of notations, some of the products are understood as dot products between vectors, or usual matrix multiplication.
Next, we formally compute the maximum likelihood estimator (MLE) for the parameter of interest $\theta$ by maximizing the log-likelihood ratio $\log( \dif \bP_{\theta}^{N}/ \dif \bP_{\theta_{0}}^{T,N}(u^{N}))$ with respect to $\theta$, and obtain the following estimator
\begin{align}
 \theta_{N}^\sharp = -\frac{\int_{0}^{T} (Q^{-1} A u^{N}, du^{N})   +\int_{0}^{T} (Q^{-1} A u^{N}, P^{N} B(u)) \dif t}
 {\int_{0}^{T}   \norm{ Q^{-1/2}A u^{N}}^2 \dif t}.
 \label{eq:EstGeneral0}
\end{align}
Note that  the estimator $\theta^\sharp$ will use the full spacial resolution of the solution $u$, unless $B$ commutes with $P^N$ which is the case for many linear equations. On the other hand, even if $B$ is nonlinear,  since $u^N$ converges to $u$ in some space, one would expect that replacing $u$ by $u^N$ would not change significantly the quality of the estimator $\theta^\sharp$. With this in mind, taking \eqref{eq:EstGeneral0} as an ansatz for our \textit{modified maximum likelihood estimators} (mMLEs), we introduce two additional degrees of freedom, $\alpha, \rho\in\bR$, and propose the following estimator for $\theta$
\begin{align}
 \widehat{\theta}_{N} &=
 	-\frac{\int_{0}^{T} (Q^{\rho} A^{\alpha} u^{N} , \dif u^{N})  +\int_{0}^{T} ( (Q^{\rho}A^{\alpha} u^{N}, P^{N} B(u^N)) \dif t}
 	{\int_{0}^{T}   \norm{Q^{\rho/2} A^{(1+\alpha)/2} u^{N}}^2  \dif t}. \label{eq:EstGeneral1}
\end{align}
\begin{remark}\label{rem:1} Several comments in order.
\begin{enumerate}[(i)]
\item If the equation \eqref{eq:mainAbstract} is linear, and $B$ has the same system of eigenfunctions $\set{h_k}_{k\in\bN}$, and thus $B$ commutes with $A$, and hence also with $Q$ and $P^N$, then $\theta^\sharp_N$ is the true MLE of $\theta$ that satisfies all desired asymptotical properties as $N\to\infty$. There is no need to modify the MLE in this case.
\item One has to be caution, since the true MLE is computed as maximum of the log-likelihood ratio over the domain $\Theta$ that supports $\theta$. For the considered SPDEs, $\theta$ is assumed to be at least positive, and in some applications $\Theta$ clearly is a finite interval, while there is no guarantee that $\widehat\theta_N$ will belong to this domain. This point is studied into details in \cite{LototskyRozovsky2017Book}. However, one may argue that if the estimator is consistent, then eventually, $\widehat{\theta}_N\in\Theta$ for large $N$, and the observer has just to discard the estimates that fall outside of $\Theta$.

\item If \eqref{eq:mainAbstract} is linear, but $B$ does not commute with $A$, then it is enough to modify $\theta^\sharp_N$ by projecting $u$ to $H^N$, i.e. take in \eqref{eq:EstGeneral1} $\alpha=1$ and $\rho=-1$. Hence, there is no need to deal with the two additional parameters $\alpha,\rho$, in order to prove consistency and asymptotical normality.
\item The additional parameters $\alpha, \rho$ are meant to take care of the nonlinear term while proving the asymptotic properties of the estimators. Usually, these free parameters are carefully chosen on case by case basis, and are dictated by the analytical properties of the solution.
\item  For some nonlinear equations, dropping out the nonlinear term in \eqref{eq:EstGeneral1}, and considering
\begin{equation} \label{eq:EstGeneral2}
\hat{\theta}^l_{N} =
 	-\frac{\int_{0}^{T} (Q^{\rho} A^{\alpha} u^{N} , \dif u^{N}) }
 	{\int_{0}^{T}   \norm{Q^{\rho/2} A^{(1+\alpha)/2} u^{N}}^2  \dif t},
\end{equation}
we will still pertain the consistency property. A lower computation complexity of the estimator $\hat\theta^l$ usually comes at the cost of a lower speed of convergence.
\item For some examples with nonlinear $B$ it useful to truncate $u$ in $B(u)$  at higher modes, and replace $u^N$ with $u^{N^p}$ for some $p\geq1$.
\end{enumerate}
\end{remark}
To establish the asymptotic properties of the estimators $\widehat{\theta}_N$ and $\hat\theta^l_N$, as usually, we write the estimators as $\theta+$`error terms', and prove that the error terms vanish. For example, using \eqref{eq:SPDEabsProj}, the estimator $\widehat\theta_N$ can be written as,
\begin{equation}\label{eq:True+Noise}
 \widehat{\theta}_{N} =  \theta -
 \frac{\sigma \int_{0}^{T} (Q^{\rho} A^{\alpha} u^{N},\dif W^{Q,N})  +\int_{0}^{T} ((Q^{\rho} A^{\alpha}  u^{N}, P^N(B(u^N)- B(u)) \dif t}
 	{\int_{0}^{T}   \norm{Q^{\rho/2} A^{(1+\alpha)/2} u^{N}}^2  \dif t}.
\end{equation}
In the next sections we will sketch the proofs, starting with the linear case.

\subsubsection{Linear Equations}
More details on estimation of $\theta$ for linear equations discussed in this section can be found in the recent monograph \cite[Chapter~6]{LototskyRozovsky2017Book}, and in the survey paper \cite{Lototsky2009Survey}.

\smallskip
\noindent\textbf{Fully diagonalizable equations.}
Let us first focus on linear case and additive noise, i.e. $B=0$, $M=0$, and zero initial data. In this case, the equation \eqref{eq:mainAbstract} is fully diagonalizable, and each Fourier mode is an Ornstein--Uhlenbeck process,
\begin{equation}\label{eq:OUk}
\dif u_k(t) +\theta \nu_k u_k\dif t = \sigma q_k \dif w_k(t), \quad k\geq 1.
\end{equation}
The estimator $\widehat{\theta}_N$ takes the form
\begin{equation}\label{eq:linearmMLE}
\widehat{\theta}_N = - \frac{\sigma \sum_{k=1}^N \int_{0}^{T} \nu_k^{\alpha} q_k^{2\rho+1} u_k \dif u_k(t)}
{\sum_{k=1}^{N} \int_0^T \nu_k^{1+\alpha}q_k^{2\rho}u_k^2(t) \dif t} =
\theta - \frac{\sigma \sum_{k=1}^N \int_{0}^{T} \nu_k^{\alpha} q_k^{2\rho+1} u_k \dif w_k(t)}
{\sum_{k=1}^{N} \int_0^T \nu_k^{1+\alpha}q_k^{2\rho}u_k^2(t) \dif t}.
\end{equation}
\begin{remark}\label{rem:curiousreader}
	A curious reader may ask why not to consider the MLE for $\theta$ by using individual Fourier mode \eqref{eq:OUk} then combine the first $N$ of these estimators into one estimator, rather than taking the MLE for the first $N$ Fourier modes at once. At intuitive level, since each Fourier mode is driven by an independent noise,  $N$ modes will contain more information than an individual one, and it makes sense to optimize over combined information rather than on individual piece of information and then combine/average them out. It turns out that indeed  $\widehat{\theta}_N$ has a higher rate of convergence, in the sense of the asymptotic normality property.  For a detailed discussion, see \cite[Chapter~6]{LototskyRozovsky2017Book}.
\end{remark}
\begin{theorem}\label{th:MLEadditive}
Assume that,  $ \nu_k\rightarrow \infty$, $q_k\rightarrow 0$, and $\alpha,\rho$ are such that
\begin{equation}\label{eq:condConv1}
\nu_{n}^{\alpha-1}q_{n}^{2\rho+2}\leq M, \ n\geq 1, \qquad \sum_{k\geq 1}\nu_{k}^{\alpha}q_{k}^{2\rho+2}=\infty,
\end{equation}
for some $M\in\bR$.
Then, $\widehat\theta_N$ is  a consistent estimator of $\theta$, i.e. $\widehat\theta_N\to\theta$ with probability one.
Moreover, if
\begin{equation}\label{eq:condConv2}
\sum_{k\geq 1}\nu_{k}^{2\alpha-1}q_{k}^{4\rho+4}=\infty,
\end{equation}
then,  $\widehat{\theta}_N$ is also asymptotically normal
\begin{equation}\label{eq:MLEnormal}
 \frac{\sum_{k=1}^{N}\nu_{k}^{\alpha}q_{k}^{2\rho+2}}{\sqrt{\sum_{k=1}^{N}\nu_{k}^{2\alpha-1}q_{k}^{4\rho+4}}} (\widehat{\theta}_N - \theta) \xrightarrow[N\to\infty]{\cD} \cN(0, \frac{2\theta}{T}).
\end{equation}
\end{theorem}
\begin{proof}
Let us take
$$
\xi_k = \int_{0}^{T}\nu_k^\alpha q_k^{\rho+1} u_k \dif w_k, \quad
\eta_k =  \int_{0}^{T} \nu_k^{1+\alpha}q_k^{\rho} u_k^2 \dif t, \quad b_n = \sum_{k=1}^{n}\bE(\eta_k),
$$
and write
\begin{equation}\label{eq:linearMLE2}
  \widehat{\theta}_N = \theta -\sigma \frac{\sum_{k=1}^{N} \xi_k}{b_N} \cdot
  \frac{\sum_{k=1}^{N}\bE(\eta_k)}{\sum_{k=1}^{N}\eta_k} =: \theta - \sigma I_1 I_2 .
\end{equation}
Consistency follows from the (strong) law of large numbers (cf. \cite[Theorem~IV.3.2]{ShiryaevBookProbability}, or \cite[Lemma~2.2]{IgorNathanAditiveNS2010} for a weaker version). The only non-trivial conditions to check are the convergence of the following series
\begin{equation}\label{eq:linearMLE3}
\sum_{k\geq1} \frac{\Var\xi_n}{b_n^2} <\infty, \qquad \sum_{k\geq1} \frac{\Var\eta_n}{b_n^2} <\infty,
\end{equation}
which will imply, respectively, that $I_1\to 0$ a.s., and $I_2\to 1$ a.s., as $N\to\infty$, and by \eqref{eq:linearMLE2} the consistency of $\widehat\theta_N$ follows.
To verify \eqref{eq:linearMLE3}, one needs to establish precise asymptotic behavior of $\Var\xi_n, \Var\eta_n$ and $b_n$, which can be done by direct evaluations. For example, one can show that (see for instance \cite[Section~2]{Lototsky2009Survey} for details on these evaluations by several methods)
\begin{align}\label{eq:ukProps}
  \bE \int_0^T  u_k^2(t) \dif t & = \frac{\sigma^{2}q_k^2}{2\theta \nu_k}\left( T - \frac{(1-e^{-2\theta \nu_k T})}{2\theta \nu_k}\right), \\
  \Var \int_0^T u_k^2\dif t & = \frac{\sigma^{4}q_{k}^{4}}{4\theta^{2}\nu_{k}^{2}}\left(\frac{2T}{\theta\nu_{k}}+\frac{2e^{-2\theta\nu_{k}T}}{\theta^{2}\nu_{k}^{2}}+\frac{e^{-4\theta\nu_{k}T}}{2\theta^{2}\nu_{k}^{2}}+
  \frac{4Te^{-2\theta\nu_{k}T}}{\theta\nu_{k}}-\frac{5}{2\theta^{2}\nu_{k}^{2}}  \right).
\end{align}
Using these, together with \eqref{eq:condConv1}, it is straightforward to show that \eqref{eq:linearMLE3} are indeed satisfied.

One way to prove asymptotic normality is to apply the  central limit theorem for martingales (cf. \cite[Theorem~6.1.4]{LototskyRozovsky2017Book} or \cite[Thoerem~5.5.4]{LiptserShiryayevBookMartingales}). By similar arguments as above, using the law of large numbers, first one shows that
$$
\lim_{N\to\infty} \frac{\sum_{k=1}^{N} \xi_k}{\sum_{k=1}^{N} \int_{0}^{T} \nu_k^{2\alpha} q_k^{2(\rho+1)} \bE(u_k^2)\dif t } = 1,
$$
with probability one, and hence by central limit theorem for martingales, we have that
$$
\lim_{N\to\infty} \frac{\sum_{k=1}^{N} \xi_k}{(\sum_{k=1}^{N} \int_{0}^{T} \nu_k^{2\alpha} q_k^{2(\rho+1)} \bE(u_k^2)\dif t)^{1/2} } \overset{\cD}{=} \cN(0,1).
$$
This, combined with \eqref{eq:linearMLE2}, after some simple evaluations, imply the asymptotic normality.
This concludes the proof.
\end{proof}

Analogously, one can study linear diagonalizable equations \eqref{eq:mainAbstract} with $A,B$ linear differential operators that commute. These class of equations were studied in the seminal works \cite{HubnerRozovskiiKhasminskii} and \cite{HuebnerRozovskii} where the spectral approach was first introduced; see also \cite{Huebner1993PhD}. In \cite{HuebnerRozovskii}, the authors showed that the consistency of $\widehat\theta_N$, as $N\to\infty$, depends on the order of the operators $A$ and $B$, and it holds true, if and only if
\begin{equation}\label{eq:ordersAB}
\textrm{ord}(A)\geq \frac12 (\textrm{ord}(\theta A + B) -d).
\end{equation}
This is dictated by the fact that the measures $\bP_\theta, \theta>0$, generated by the solution $u$ of \eqref{eq:mainAbstract}, are mutually singular if and only if \eqref{eq:ordersAB} is satisfied. In \cite{HuebnerRozovskii,LototskyRozovsky2017Book} the authors also study the asymptotic efficiency of MLEs.
If \eqref{eq:ordersAB} does not hold true, then the measures $\bP_\theta, \theta>0$, are absolutely continuous to each other. This is the case with the Example~II from Section~\ref{sec:Intro}, for d=1. The only way to obtain consistency and asymptotic normality  of $\widehat\theta_N$ in this case is to consider large time asymptotics $T\to\infty$, small noise asymptotics $\sigma\to0$, or a combination of large times, small noise and large spectral resolution $N\to\infty$.

Using similar arguments and methods, one can study fully diagonalizable SPDEs with several unknown parameters (in front of $A$ and $B$); cf. \cite{Huebner1993PhD,Huebner1997}.

\smallskip\noindent
\textbf{Almost diagonalizable equations.} Analogous results to fully diagonalizable case hold true if $B$ is linear, but does not commute with $A$.
While the asymptotic behavior of $\widehat{\theta}_N$ given by \eqref{eq:EstGeneral1}, with $\alpha=1,\rho=-1$, remains the same, the proofs become much more technical. These classes of equations were studied in \cite{Lototskiy1996PhD,HuebnerLototskyRozovskii97,Huebner1997,LototskyRozovskii1999,LototskyRozovskii2000,Lototsky2003}.


\subsubsection{Nonlinear Equations} The never diagonalizable nonlinear equations \eqref{eq:mainAbstract}, with $B$ being the nonlinear part, and driven by an additive noise ($M=0$), still can be studied within the spectral approach. The key idea in this case, introduced in \cite{IgorNathanAditiveNS2010}, is to split the solution in its linear and nonlinear part, a technic often used in studying PDEs and SPDEs. Namely, the solution $u$ of equation \eqref{eq:mainAbstract} is written as $u=\bar{u} + v$, where $\bar{u}$  solves the equation $\dif \bar{u} +\theta A\bar{u}\dif t = \sigma \dif W^Q(t)$, $\bar{u}(0)=u_0$, and $v$ is the solution of equation
\begin{equation}\label{eq:NSEv}
 \dif v + \theta A v\dif t = - Bu\dif t, \quad t>0, \  v(0)=0.
\end{equation}
Recall, that consistency is proved by showing that the second term in the right hand side of \eqref{eq:True+Noise} vanishes. Using the splitting argument, this reduces to show that
\begin{align}\label{eq:NSEvanish}
J_1&:=\frac{\int_{0}^{T} (Q^{\rho} A^{\alpha} \bar u^{N},\dif W^{Q,N})}{\int_{0}^{T}   \norm{Q^{\rho/2} A^{(1+\alpha)/2} \bar u^{N}}^2  \dif t} \to 0, \quad
J_2:=\frac{\int_{0}^{T} (Q^{\rho} A^{\alpha} v^{N},\dif W^{Q,N})}{\int_{0}^{T}   \norm{Q^{\rho/2} A^{(1+\alpha)/2} \bar u^{N}}^2  \dif t} \to 0, \\
J_3&:=\frac{\int_{0}^{T}   \norm{Q^{\rho/2} A^{(1+\alpha)/2} \bar (\bar u^N + v^N)}^2  \dif t}{\int_{0}^{T}   \norm{Q^{\rho/2} A^{(1+\alpha)/2} \bar u^{N}}^2  \dif t} \to 1, \quad
J_4:=\frac{\int_{0}^{T} ((Q^{\rho} A^{\alpha}  u^{N}, P^N(B(u^N)- B(u))) \dif t }
 	{\int_{0}^{T}   \norm{Q^{\rho/2} A^{(1+\alpha)/2} \bar{u}^{N}}^2  \dif t}\to 0,
\end{align}
where the converges is either in a.s. sense or in probability. Note that $J_1\to 0$ is already covered by Theorem~\ref{th:MLEadditive}.
Another key point in using the splitting method, is that the nonlinear part $v$ usually is slightly more regular that $\bar u$, and hence its Fourier modes will vanish faster to zero. This observation, used carefully, together with choosing the appropriate $\alpha$ and $\rho$, allows to show that $J_2\to0$, and $J_3\to1$. By akin arguments, and using the specific form of the nonlinear part, one shows that $J_4\to0$.

Despite the fact that the steps described above sound reasonable, each nonlinear equation has to be studied separately. To best of our knowledge, the only nonlinear equations  studied in the current literature are the 2D Navier--Stokes Equations; Example~III from Section~\ref{sec:Intro}. For simplicity of wirtting, let us assume that $q_k=\nu_k^{-\gamma}$. Then, assuming that $\gamma>1$ and $\alpha+1> 2\gamma(\rho+1)$, it was proved in \cite{IgorNathanAditiveNS2010} that $\widehat{\theta}_N$  and $\hat \theta_N^l$ are consistent estimators for $\theta$, as $N\to\infty$, in the context of \eqref{eq:NVS}.
As one may expect, the proposed method, as well as the technical proves from the 2D Navier--Stokes Equations should care out for Burgers equation
$$
\dif u(t,x) + \left(-\theta u_{xx}(t,x) + \frac{1}{2}\pd{x}(u^2(t,x)) \right) \dif t = \sigma \dif W^Q(t,x), \quad t>0,
$$
where $x\in[0,\pi]$, $u(0,x)=u_0$, and $u(t,0)=u(t,\pi)=0$. We leave the detailed analysis of this equation to future studies.

\subsubsection{Fractional noise} A natural problem to consider is parameter estimation for SPDEs driven by a fractional Brownian noise. It turns out that the spectral approach, and the MLEs, can be successfully applied to such equations that a fully diagonalizable and driven by an additive noise. If $w^H(t), t\geq0$, is a fractional Brownian motion with Hurst parameter $H$, then $w^H$ is not a martingale, unless $H=1/2$. However, it is well known that $\int_0^t k_H(t,s)dw^H(s)$, where $s^{\frac12 -H}(t-s)^{\frac12-H}$, is a martingale. Then, it is enough to integrate against $k_H$ the Fourier modes $u_k$ that follow  the dynamics
$$
\dif u_k(t) + (\theta \nu_k + \mu_k )u_k(t)\dif t = \sigma q_k\dif w^H_k(t),
$$
where $w^H_k, k\in\bN$, are independent fractional Brownian motions, and $\mu_k, k\in\bN$, are the eigenvalues of $B$. After this transformation, being in the martingale setup, in \cite{IgorSergeyJan2008} the authors  followed the procedure of deriving the MLEs describe above with $\alpha=1, \rho=-1$, thanks to the Girsanov type theorem developed in \cite{KleptsynaBreton2002,KleptsynaLeBretonRoubaud2000}. Same problem but in small noise asymptotic regime was studied in \cite{PrakasaRao2004}. Since the kernel $k^H$ is singular, proving that the error terms in \eqref{eq:True+Noise} vanish is not an obvious task, and it remains an open problem for almost diagonalizable or nonlinear SPDEs.

\subsection{Time dependent drift}
Let us consider a fully diagonalizable parabolic SPDEs \eqref{eq:mainAbstract} driven by an additive space-time white noise ($q_k=1$), and where $\theta$ is time-dependent. The non-parametric statistics for finite dimensional diffusion is a well developed field (cf. \cite{KutoyantsBook2004}). Similar to MLEs some of the methods from SODEs can be adapted to the SPDE setup, again, using the spectral approach. A kernel based interpolation method was employed in \cite{HuebnerLototsky2000a}. Let $R(t)$ be a compactly supported kernel of order $K\geq 1$, that is, $R$ has compact support, $\int_{\bR}R(t)\dif t=1$, and
$\int_{\bR}t^{j}R(t)\dif t=0$ for $j=1,\ldots,K$. Let $(v_n)_{n\geq 1}$, and $(a_n)_{n\geq1}$ be sequences of positive real numbers, monotonically decreasing and convergent to zero. Define the inverse cut-off function of $u_k$ as follows
\begin{align*}
U_{k,N}(t):=\left\{
\begin{array}{lll}
1/u_{k}(t),\quad |u_{k}(t)|>v_{N},\\
1/v_{N},\quad |u_{k}(t)|\leq v_{N}.
\end{array}\right.
\end{align*}
Then, the kernel based estimator for $\theta(t)$ is defined as
\begin{align}\label{eq:kernalestimator}
\widetilde{\theta}_{N}(t)=\frac1{h_{N}\sum_{k=1}^{N}\nu_{k}}\sum_{k=1}^{N}\int_{0}^{T}R\left(\frac{s-t}{h_{N}}\right)U_{k,N}(s)(\dif u_{k}(s)-\mu_{k}u_{k}(s)\dif s),
\quad t\in[0,T], \ N\geq1.
\end{align}
In \cite{HuebnerLototsky2000a} it was proved that the $\widetilde{\theta}_{N}(t)$ converges in mean-square sense to $\theta(t)$. Also within the spectral approach, in \cite{HuebnerLototsky2000} the authors propose a sieve estimator for time dependent viscosity coefficient $\theta(t)$.

\subsection{Simple multiplicative noise} Let us focus on the Example~IV, where the noise is multiplicative, but driven just by an one dimensional Wiener process. The results of this subsection can be easily extended to finite dimensional noise.  This type of equations are remarkably  interesting from inference point of view, and underpin one more time the singular nature of these problems. On the one hand, \eqref{eq:heatSimpleM} is diagonalizable, in the sense that the Fourier modes are decoupled, and follow the dynamics of a geometric Brownian motion
$$
\dif u_k(t) + \theta \nu_k u_k(t)\dif t = \sigma u_k(t)\dif w(t), \quad t>0, u_k(0) = (u_0,h_k),
$$
where in this case $h_k=\sqrt{2/\pi}\sin(k x)$ and $\nu_k=k^2$ are the eigenfunctions and the corresponding eigenvalues of the $-\Delta$ on $[0,\pi]$ with zero boundary conditions.
Thus, this SPDE falls under the spectral approach. On the other hand, the general mMLE method described in Section~\ref{sec:mMLE} can not be applied directly, and even formally one can not apply Girsanov transformation to the first $N$ Fourier modes considered together. However, we can use each individual Fourier mode to obtain an MLE for $\theta$ (in contrast to the Remark~\ref{rem:curiousreader})
\begin{equation}\label{eq:MLEsimple}
	\check\theta_N = -\frac{1}{\nu_k T} \int_0^T \frac{\dif u_k}{u_k} = \frac{1}{\nu_k T} \ln \frac{u_k(T)}{u_k(0)} + \frac{\sigma^2}{2\nu_k}.
\end{equation}
It is easy to show that $\check\theta_N$ is consistent and asymptotically normal as $N\to\infty$.

What is even more interesting, since each Fourier mode contains the same noise factor $w(t)$, taking any two nontrivial modes allows to eliminate the noise altogether, and to solve for the unknown parameter $\theta$ explicitly
$$
\theta = \frac{1}{T(\nu_m-\nu_k)} \ln\frac{u_k(T) u_m(0)}{u_k(0)u_k(T)}, \quad k\neq m.
$$
Hence, once any two Fourier modes are observed at any time point $T$, the parameter can be found exactly, without any statistical procedure.
In \cite{CialencoLototsky2009}, we call such `estimators' closed-from exact estimators. The fractional noise counterpart of these class of SPDEs is studied in \cite{Cialenco2010}.

\subsection{Hypothesis testing}
Most of the literature on statistical inference for (parabolic) SPDEs concerns the parameter estimation problem in various setups and forms, with only exception  \cite{CXu2014,CXu2015} where the authors study the simple hypothesis problems for the drift coefficient $\theta$ for the following SPDE
\begin{equation}\label{eq:mainSPDE-ht}
\dif u(t,x) + \theta (-\Delta)^\beta u(t,x)\dif t = \sigma \sum_{k\in\bN} \lambda_k^{-\gamma}h_k(x)\dif w_k(t), \quad t\in[0,T], \ u(0,x) = u_0, \ x\in G,
\end{equation}
where $\theta>0$, $\beta\geq1, \ \gamma \geq 0$, $\sigma\in\bR_+$, $\Delta$ is endowed with zero boundary conditions on the bounded domain $G\in\bR^d$, and where $\lambda_k$ are the eigenvalues of the $\sqrt{-\Delta}$.
In this case, the likelihood ratio takes the form
\begin{align}
L(\theta_0,\theta;u^N)=
\exp\Big(-\frac{\theta-\theta_0}{\sigma^{2}}  \sum_{k=1}^N\lambda_k^{2\beta+2\gamma}\big(\int_0^T  u_k(t)du_k(t)
 +\frac{1}{2}(\theta+\theta_0)\lambda_k^{2\beta}\int_0^Tu_k^2(t)dt\big)\Big), \label{eq:RadonNikodymUn}
\end{align}
that yields the (true) MLE for $\theta$ as
\begin{equation}\label{eq:MLE-UN}
\widehat{\theta}_T^N = -\frac{\sum_{k=1}^{N}\lambda_k^{2\beta+2\gamma}\int_0^T u_k(t)du_k(t)}{\sum_{k=1}^{N}\lambda_k^{4\beta+2\gamma}\int_0^T u_k^2(t)dt}.
\quad N\in\bN.
\end{equation}
Although in \cite{CXu2014,CXu2015} both regimes, large times $T\to\infty$ and fine spacial resolution $N\to\infty$ are investigated, we will focus here only on latter, $N\to\infty$.  The estimator $\widehat\theta_N$ is consistent, and asymptotically normal with rate of convergence $N^{\beta/d+1/2}$.
Assume that $\theta$ can take only two values $\theta_0<\theta_1$, and consider the simple hypothesis
$$
\mathscr{H}_0  : \ \theta=\theta_0,  \quad \textrm{vs} \quad \mathscr{H}_1 : \ \theta=\theta_1.
$$
With MLE and its asymptotic properties at hand, it is easy to establish a Neyman-Pearson type lemma.
Indeed, if   $c_\alpha$ is a real number such that
$\bP^{N,T}_{\theta_0}(L(\theta_0,\theta_1,u^N)\ge c_{\alpha})=\alpha$, with $\alpha\in(0,1)$ being the significance level. Then,
$R^*:=\{u^N: L(\theta_0,\theta_1,u^N)\ge c_{\alpha}\}$,
is the most powerful rejection region in the class $\mathcal{K}_{\alpha}:=\left\{R\in\cB(C([0,T];\bR^N)): \bP^{N}_{\theta_0}(R)\leq\alpha\right\}$, i.e.
$\bP^{N,T}_{\theta_1}(R)\le\bP^{N,T}_{\theta_1}(R^*)$, for all $R\in\mathcal{K}_{\alpha}$. The problem is that the constant $c_\alpha$ can not be computed explicitly. To address this issue and to find a Likelihood Ratio type test, we approximate $c_\alpha$ by an appropriately chosen sequence $c_\alpha(N)$, and apply the concept of asymptotically the most powerful test introduced in \cite{CXu2014};  the rejection region $(\widetilde{R}_N)\in\widetilde{\mathcal{K}}$ is \textit{asymptotically the most powerful},  in the class $\widetilde{\mathcal{K}}$, as $N\to\infty$, if
\begin{align}\label{eq:asym-def-N}
\underset{N\to\infty}{\liminf} \frac{1-\bP^{N,T}_{\theta_1}(R_N)}{1-\bP^{N,T}_{\theta_1}(\widetilde{R}_N)}\ge1,\qquad \textrm{ for all } (R_N)\in \widetilde{\mathcal{K}}.
\end{align}
To compensate for the asymptotic nature of this setup, one has to shrink the class $\cK_\alpha$, or more precisely its asymptotic version, by considering tests that have a certain rate of convergence of the Type~I error. Namely, for $\delta\in\bR$, we put
$\widehat{\mathcal{K}}_{\alpha}(\delta):=\left\{(R_N): \limsup_{N\to\infty}\left(\bP^{N}_{\theta_0}(R_N)-\alpha\right)\sqrt{M}\leq{\alpha}_1(\delta)\right\}$,
where $M=\sum_{k=1}^N\lambda_k^{2\beta}$, and $\alpha_1(\delta)$ is an explicitly computable constant (see \cite{CXu2014}, formula (3.15)), and
$$
\widehat{R}_N^\delta=\left\{u^N: L(\theta_0,\theta_1,u^N)\geq {c}^\delta_{\alpha}(N)\right\},
$$ with
$$
{c}^\delta_{\alpha}(N)=\exp\left(-\frac{(\theta_1-\theta_0)^2 TM}{4\theta_0}+\frac{(\theta_1-\theta_0)^2N}{8\theta_0^2 }-\frac{\sqrt{TM}(\theta_1^2-\theta_0^2)}{\sqrt{8\theta_0^3}}q_{\alpha} -\frac{\sqrt{T}(\theta_1^2-\theta_0^2)}{\sqrt{8\theta_0^3}}\delta\right).
$$
Then, the rejection region $(\widehat{R}_N^\delta)$ is asymptotically the most powerful in the class $\widehat{\mathcal{K}}_{\alpha}(\delta)$.
The main challenge is to identify the `right class' of tests, that requires an exact control of the power of the tests as $N\to\infty$. The proofs are rooted in the theory of sharp large deviations adapted to the case of large number of Fourier modes.

In \cite{Markussen2003} the author also studies a simple hypothesis testing problem, in discrete sampling, by testing if an SPDE is parabolic or hyperbolic, also within the spectral approach.

It would be fair to say that this is just the first step towards hypothesis testing problems, and goodness of fit test for SPDEs, with many open problems left.

\section{Discrete sampling}\label{sec:discrete}
The literature on parameter estimation for discretely sampled SPDEs  is limited, while in most applications, the observer will measure one realization of the solution $u$ only at some discrete points $(t_j,x_k)$ in time and/or space. Of course, one way to deal with discretely sampled data, is to discretize or approximate the (m)MLEs using the available discrete data, and show that the statistical properties are preserved. We refer the reader to \cite{PrakasaRao2002,PrakasaRao2003} for some results on this approach.

In \cite{PiterbargRozovskii1997} the authors develop MLE type estimators for $\theta$, within the spectral approach, by assuming that the Fourier modes are sampled at some discrete time points on a finite time interval, and under some additional technical assumptions show that these estimators are consistent and asymptotically normal as the mesh size of the time partition goes to zero.

Note that, if we assume that the solution itself is observed at some space-time grid points, one needs to approximate additionally the Fourier modes. To best of our knowledge, a rigorous asymptotic analysis of this is still an open question. The closest to this approach is \cite{Markussen2003} where the author considers a parabolic (or hyperbolic) SPDE on $[0,1]$ driven by an additive noise, and constructs an approximated MLE, by assuming that the solution is observed at a finite and fixed number $m$ of spacial points, and at discrete time points $t_k=\delta k$,  $k=1,\ldots,n$, and for some $\delta>0$. The approximated MLE is constructed through the Fourier modes, and it is proved that these estimators are consistent and asymptotically normal as $n\to\infty$, while $\delta,m$ are fixed.

As already mentioned, by its nature the spectral approach requires that the Fourier decomposition is performed with respect to the basis formed by the eigenfunctions of the operator $A$. If $A$ is a differential operator, then essentially one has to deal with bounded domains.

Next, we will discuss some results related to discrete sampling of SPDE \eqref{eq:mainAbstract} that are obtained without assuming the spectral approach.
It turns out that for some classes of SPDEs, to estimate $\theta$ and $\sigma$, it is enough to observe the solution at one time instant and discretely on a spacial grid of a finite interval, with mesh diameters going to zero, or just at one spacial point, and over a time-grid interval. Namely, we will focus on two sampling schemes
\begin{enumerate}[(A)]
	\item \textit{Fixed time and discrete space.} For a fixed instant of time $t>0$, and given interval $[a,b]\subset G\subset\bR$, the solution $u$ is observed at points  $(t,x_j), \ j=1,\ldots,m$, with $x_{j}=a+(b-a)j/m,\quad j=0,1,\ldots,m$.
	\item \textit{Fixed space and discrete time.}  For a fixed  $x$ from the interior of $G\subset\bR$, and given time interval $[c,d]\subset(0,+\infty)$, the solution $u$ is observed at points $\{(t_{i},x), \, i=1,\ldots,n\}$, where $t_{i}:=c+ (d-c)i/n,\ i=0,1,\ldots,n$.
\end{enumerate}
For simplicity of writing, we assume that the sampling points form a uniform grid, but generally speaking the results hold true assuming only that the mesh size of the grid goes to zero.
We will use the notation  $\Upsilon^m(a,b)=\set{a_j \mid a_j= a+(b-a)j/m,\ j=0,1,\ldots,m}$ for the uniform partition of size $m$ of a given  interval $[a,b]\subset\bR$. For a given stochastic process $X$ on some interval $[a,b]$, and $p\geq 1$, we define the $p$-variation as
$$
\mV^p(X; [a,b]) :=  \lim_{m\to\infty}\mV^p_m(X; [a,b]), \quad \bP-\textrm{a.s.}, \quad \mV^p_m(X; [a,b]) := \sum_{j=1}^{m} |X(t_j) - X(t_{j-1})|^p,
$$
and $\mV^p_\bP(X; [a,b])$ will denote the $p$-variation when the limit is understood in probability sense; sometimes we will will simply write $\mV^p(X)$, and $\mV^p_m(X)$.

The main idea behind this method is similar to the argument of estimating the volatility coefficient in finite dimensional diffusions through quadratic variation arguments. In fact, estimating $\sigma$ for SPDEs \eqref{eq:mainAbstract} can be achieved by this approach. Estimating the drift $\theta$ is more delicate, and one has to find the correct, and exact,  $p$-variation to be used.

In \cite{PospivsilTribe2007} (see also \cite{Pospivsil2005PhD}) the authors explore this idea, for the stochastic heat equation on whole real line
\begin{align*}
\dif u(t,x)-\theta\Delta u(t,x)\dif t = \sigma(u(t,x))\dif W(t,x), \quad t>0,\quad x\in \bR,
\end{align*}
with $u(0,x)=0,\ x\in\bR$, and were  $\theta>0$ is the parameter of interest, $\sigma:\bR\to\bR$, and $W$ is the space-time white noise on $(0,\infty)\times\bR$. Consider the sampling scheme (A), with the partition $\Psi^m(a,b)$ for some fixed $a,b\in\bR$. It was proved in \cite{PospivsilTribe2007} that
\begin{align*}
\wh{\theta}_{m,t}:=\frac{b-a}{2m}\frac{\sum_{j=1}^{m}\sigma^{2}(u(t,x_{j}))}{\sum_{j=1}^{n}(u(t,x_{j})-u(t,x_{j-1}))^{2}}
\end{align*}
is a weakly consistent estimator\footnote{An estimator is weakly consistent if it converges to the true parameter in probability.}  of $\theta$ as $m\to\infty$. Similarly, for some fixed $0<c,d\leq T$ and $x\in\bR$, assuming that the solution is sampled by scheme (B), the estimator
\begin{align*}
\wh{\theta}_{n,x}:=\frac{3(d-c)}{n\pi}\frac{\sum_{j=1}^{n}\sigma^{4}(u(t_{j},x))}{\sum_{j=1}^{n}(u(t_{j},x)-u(t_{j-1},x))^{4}}
\end{align*}
is weakly consistent as $n\to\infty$. As one may guess, $\wh\theta_{m,t}$ comes from the computations of the quadratic variation, while $\wh\theta_{n,x}$ is due to the fourth-variation of the solution. Moreover, assuming that $\sigma(u)=\sigma u$, for some positive constant $\sigma$ (i.e. Example~V, with $G=\bR$, and driven by space-time white noise), and assuming that $\theta$ is known, then the estimators
\begin{align*}
\wh{\sigma}^{2}_{m,t}=\frac{2m\theta\sum_{j=1}^{m}(u(t,x_{j})-u(t,x_{j-1}))^{2}}{(b-a)\sum_{j=1}^{m}u^{2}(t,x_{j})}, \quad
\wh{\sigma}^{2}_{n,x}&:=\sqrt{\frac{n\theta\pi\sum_{i=1}^{n}(u(t_{i},x)-u(t_{i-1},x))^{4}}{3(d-c)\sum_{i=1}^{n}u^{4}(t_{i},x)}}
\end{align*}
are weakly consistent estimators of $\sigma$, respectively when $m\to\infty$, and $n\to\infty$. The asymptotic normality of these estimators remains an open problem.

Two recent independent studies \cite{CialencoHuang2017,BibingerTrabs2017} were devoted to the stochastic heat equation driven by an additive space-time white noise.

In  \cite{CialencoHuang2017} the authors further explore the quadratic variation method, starting with a simple and intuitively clear observation: the $p$-variation of a stochastic process is invariant with respect to smooth perturbations. Hence, if the $p$-variation of a process $X$ can be computed by an explicit formula, and the parameter of interest enters non-trivially into this formula, one can derive consistent estimators of this parameter. However, since the $p$-variation of the perturbed process $X+Y$ remains the same, given that $Y$ is smooth enough, then the same estimator remains consistent assuming that $X+Y$ is observed. Analogous arguments remain valid for asymptotic normality property; see \cite[Proposition~2.1]{CialencoHuang2017}. Hence, it remains to establish such representations for the solutions of the considered SPDEs. This can be done for equation \eqref{eq:heatAdditive},  driven by space-time white noise, with $G$ being either a finite interval or the whole real line ($d=1$).

Assume that the solution $u$ is sampled according to the sampling scheme (A), for some fixed $t>0$. Then, one can show that the  following estimators for $\theta$ (assuming $\sigma$ is known) and $\sigma^2$ (assuming $\theta$ is known), respectively,
\begin{align*}
\wh{\theta}_{m,t}  :=\frac{(b-a)\sigma^{2}}{2\sum_{j=1}^{m}(u(t,x_{j})-u(t,x_{j-1}))^{2}}, \quad
\wh{\sigma}^{2}_{m,t} :=\frac{2\theta}{b-a}\sum_{j=1}^{m}(u(t,x_{j})-u(t,x_{j-1}))^{2},
\end{align*}
are consistent and asymptotically normal with rate of convergence $\sqrt{m}$.

Similarly, using the sampling scheme (B),  the estimators
\begin{align}
\wh{\theta}_{n,x} :=\frac{3(d-c)\sigma^{4}}{\pi\sum_{i=1}^{n}(u(t_{i},x)-u(t_{i-1},x))^{4}}, \quad
\wh{\sigma}^2_{n,x} :=\sqrt{\frac{\theta\pi}{3(d-c)}\sum_{i=1}^{n}(u(t_{i},x)-u(t_{i-1},x))^{4}},
\end{align}
are also consistent and asymptotically normal with rate of convergence $\sqrt{n}$. It is interesting to note that the estimators and the rate of converge remain the same for both, bounded and unbounded domains, although the proof of asymptotic results differ.

Finally, let us consider a variation of equation \eqref{eq:heatAdditive}, that was studied in \cite{BibingerTrabs2017}
\begin{align*}
\dif u(t,x)&=\left(\theta_{2}\frac{\partial^{2}}{\partial x^{2}}u(t,x)+\theta_{1}\frac{\partial}{\partial x}u(t,x)+\theta_{0}u(t,x) \right)\dif t+\sigma(t)\dif W(t,x),\quad  t>0, \ x\in [0,1]\\
u_{0}&=\xi.
\end{align*}
and with zero boundary condition, and where $\sigma$ is an $\alpha$-H\"older continuous function with $\alpha\in(1/2,1]$.  Assume the sampling scheme (B) over $\Psi_n(0,T)$. In \cite{BibingerTrabs2017} the authors construct an estimator $\textrm{RV}_{n,x}$ of the integrated square volatility and showed that
 \begin{align*}
\textrm{RV}_{n,x} := \frac{1}{\sqrt{n T}} \sum_{i=1}^{n}(u(t_{i},x)-u(t_{i-1},x))^{2} \xrightarrow[n\to\infty]{\bP}
\frac{e^{-x\theta_{1}/\theta_{2}}}{\sqrt{\theta_2\pi}} \int_{0}^{T}\sigma^2(t)\dif t,
\end{align*}
and also established an asymptotic normality result for this estimator. In contrast to \cite{CialencoHuang2017} where the authors use elements of  Malliavin calculus, the methods used in \cite{BibingerTrabs2017} are rooted in the mixing theory for Gaussian time series.

\section{Other methods and results}
In this section we will briefly report on some existing results relevant to the identification of the drift and volatility for SPDE.

\smallskip\noindent\textbf{Trajectory fitting estimators.}
Note that as such, the spectral approach just reduces the original infinite-dimensional problem to a finite dimensional one, and hence, one can try to apply any available method from statistical inference for SODEs to the projected system, not necessarily based on MLEs. One such attempt, proposed in \cite{CialencoGongHuang2016}, was to investigate the applicability of the so called trajectory fitting estimators for ergodic processes first introduced by Kutoyants~\cite{Kutoyants1991} (see also \cite[Section 1.3 \& Section 2.3]{KutoyantsBook2004}),  that are an  analog of the least squares estimators widely used in time-series analysis.

\smallskip
\noindent \textbf{Filtering.} Generally speaking results on filtering in the context of SPDEs are rather scarce. Using spectral approach, and assuming that $\theta(t)$ follows an unobservable Ito diffusion, Lototsky~\cite{Lototsky2004} derives an optimal filter for $\theta$, which can be viewed as a generalization of Kalman-Bucy filter. Using different methods and technics than those mentioned above, in as series of papers Aihara et al. \cite{Aihara1991a,Aihara92,Aihara1998,Aihara1988,Aihara1998a,AiharaBagchi1989} study a non-parametric estimation problem of a space dependent $\theta$, combined with a filtering problem by assuming that the observations are $y(t) = \int_0^t F(u(s))\dif s + w(t)$, where where $u$ is the solution of the corresponding parabolic SPDE driven by an additive noise, $F$ is an operator on $H$ with finite-dimensional range, and $w$ is a finite-dimensional Brownian motion.

\smallskip
\noindent \textbf{Bayesian Inference} for SPDEs is another area with few existing results \cite{Bishwal2002,Bishwal1999,PrakasaRao2000} all within the spectral approach.

\smallskip
\noindent \textbf{Concluding remarks.} While this survey is dense and most of the proofs have been omitted, the overall statistical inference for SPDEs is yet to become a mature field with many practically important problems remaining open. Besides some open problems mentioned above, for example, there is little known about estimation of the drift of simplest SPDE, the stochastic heat equation, driven by a multiplicative space-time noise; equation \eqref{eq:heatMultiple}, Example~V. The inference for SPDEs driven by non-Gaussian noise is another open field for investigations.

\section*{Acknowledgments}
The author is grateful to Yicong Huang for reading the initial draft of this paper and making a series of useful comments that improved greatly the final manuscript.

\bibliographystyle{alpha} 

\begin{thebibliography}{KLBR00}

\bibitem[AB89]{AiharaBagchi1989}
S.~I. Aihara and A.~Bagchi.
\newblock Infinite-dimensional parameter identification for stochastic
  parabolic systems.
\newblock {\em Statist. Probab. Lett.}, 8(3):279--287, 1989.

\bibitem[Aih91]{Aihara1991a}
S.~I. Aihara.
\newblock Parameter identification for stochastic parabolic systems.
\newblock In {\em Systems and control}, pages 1--12. Mita, 1991.

\bibitem[Aih92]{Aihara92}
S.~I. Aihara.
\newblock Regularized maximum likelihood estimate for an infinite-dimensional
  parameter in stochastic parabolic systems.
\newblock {\em SIAM J. Control Optim.}, 30(4):745--764, 1992.

\bibitem[Aih98a]{Aihara1998a}
S.~I. Aihara.
\newblock Consistency property of extended least-squares parameter estimation
  for stochastic diffusion equation.
\newblock {\em Systems Control Lett.}, 34(5):249--256, 1998.

\bibitem[Aih98b]{Aihara1998}
S.~I. Aihara.
\newblock Identification of a discontinuous parameter in stochastic parabolic
  systems.
\newblock {\em Appl. Math. Optim.}, 37(1):43--69, 1998.

\bibitem[AS88]{Aihara1988}
S.~I. Aihara and Y.~Sunahara.
\newblock Identification of an infinite-dimensional parameter for stochastic
  diffusion equations.
\newblock {\em SIAM J. Control Optim.}, 26(5):1062--1075, 1988.

\bibitem[Bis99]{Bishwal1999}
J.~P.~N. Bishwal.
\newblock Bayes and sequential estimation in {H}ilbert space valued stochastic
  differential equations.
\newblock {\em J. Korean Statist. Soc.}, 28(1):93--106, 1999.

\bibitem[Bis02]{Bishwal2002}
J.~P.~N. Bishwal.
\newblock The {B}ernstein-von {M}ises theorem and spectral asymptotics of
  {B}ayes estimators for parabolic {SPDE}s.
\newblock {\em J. Aust. Math. Soc.}, 72(2):287--298, 2002.

\bibitem[BT17]{BibingerTrabs2017}
M.~Bibinger and M.~Trabs.
\newblock Volatility estimation for stochastic {PDE}s using high-frequency
  observations.
\newblock {\em Preprint, arXiv:1710.03519}, 2017.

\bibitem[CGH11]{IgorNathanAditiveNS2010}
I.~Cialenco and N.~Glatt-Holtz.
\newblock Parameter estimation for the stochastically perturbed
  {N}avier-{S}tokes equations.
\newblock {\em Stochastic Process. Appl.}, 121(4):701--724, 2011.

\bibitem[CGH16]{CialencoGongHuang2016}
I.~Cialenco, R.~Gong, and Y.~Huang.
\newblock Trajectory fitting estimators for {SPDE}s driven by additive noise.
\newblock {\em Forthcoming in Statistical Inference for Stochastic Processes
  DOI:10.1007/s1120}, 2016.

\bibitem[CH17]{CialencoHuang2017}
I.~Cialenco and Y.~Huang.
\newblock A note on parameter estimation for discretely sampled {SPDE}s.
\newblock {\em Preprint, arXiv:1710.01649}, 2017.

\bibitem[Cho07]{ChowBook}
P.~Chow.
\newblock {\em Stochastic partial differential equations}.
\newblock Chapman \& Hall/CRC Applied Mathematics and Nonlinear Science Series.
  Chapman \& Hall/CRC, Boca Raton, FL, 2007.

\bibitem[Cia10]{Cialenco2010}
I.~Cialenco.
\newblock Parameter estimation for {SPDE}s with multiplicative fractional
  noise.
\newblock {\em Stoch. Dyn.}, 10(4):561--576, 2010.

\bibitem[CL09]{CialencoLototsky2009}
I.~Cialenco and S.~V. Lototsky.
\newblock Parameter estimation in diagonalizable bilinear stochastic parabolic
  equations.
\newblock {\em Stat. Inference Stoch. Process.}, 12(3):203--219, 2009.

\bibitem[CLP09]{IgorSergeyJan2008}
I.~Cialenco, S.~V. Lototsky, and J.~Posp{\'{\i}}{\v{s}}il.
\newblock Asymptotic properties of the maximum likelihood estimator for
  stochastic parabolic equations with additive fractional {B}rownian motion.
\newblock {\em Stoch. Dyn.}, 9(2):169--185, 2009.

\bibitem[CX14]{CXu2014}
I.~Cialenco and L.~Xu.
\newblock A note on error estimation for hypothesis testing problems for some
  linear {SPDE}s.
\newblock {\em Stoch. Partial Differ. Equ. Anal. Comput.}, 2(3):408--431, 2014.

\bibitem[CX15]{CXu2015}
I.~Cialenco and L.~Xu.
\newblock Hypothesis testing for stochastic {PDE}s driven by additive noise.
\newblock {\em Stochastic Process. Appl.}, 125(3):819--866, March 2015.

\bibitem[DPZ92]{DaPrato}
G.~Da~Prato and J.~Zabczyk.
\newblock {\em Stochastic equations in infinite dimensions}, volume~44 of {\em
  Encyclopedia of Mathematics and its Applications}.
\newblock Cambridge University Press, Cambridge, 1992.

\bibitem[GM02]{GoldysMaslowski2002}
B.~Goldys and B.~Maslowski.
\newblock Parameter estimation for controlled semilinear stochastic systems:
  identifiability and consistency.
\newblock {\em J. Multivariate Anal.}, 80(2):322--343, 2002.

\bibitem[Hai09]{HairerBook2009}
M.~Hairer.
\newblock {\em Introduction to {S}tochastic {PDEs}}.
\newblock Unpublished lecture notes, 2009.

\bibitem[HKR93]{HubnerRozovskiiKhasminskii}
M.~Huebner, R.~Khasminskii, and B.~L. Rozovskii.
\newblock Two examples of parameter estimation for stochastic partial
  differential equations.
\newblock In {\em Stochastic processes}, pages 149--160. Springer, New York,
  1993.

\bibitem[HL00a]{HuebnerLototsky2000a}
M.~Huebner and S.~V. Lototsky.
\newblock Asymptotic analysis of a kernel estimator for parabolic {SPDE}'s with
  time-dependent coefficients.
\newblock {\em Ann. Appl. Probab.}, 10(4):1246--1258, 2000.

\bibitem[HL00b]{HuebnerLototsky2000}
M.~Huebner and S.~V. Lototsky.
\newblock Asymptotic analysis of the sieve estimator for a class of parabolic
  {SPDE}s.
\newblock {\em Scand. J. Statist.}, 27(2):353--370, 2000.

\bibitem[HLR97]{HuebnerLototskyRozovskii97}
M.~Huebner, S.~V. Lototsky, and B.~L. Rozovskii.
\newblock Asymptotic properties of an approximate maximum likelihood estimator
  for stochastic {PDE}s.
\newblock In {\em Statistics and control of stochastic processes (Moscow,
  1995/1996)}, pages 139--155. World Sci. Publishing, 1997.

\bibitem[HR95]{HuebnerRozovskii}
M.~Huebner and B.~L. Rozovskii.
\newblock On asymptotic properties of maximum likelihood estimators for
  parabolic stochastic {PDE}'s.
\newblock {\em Probab. Theory Related Fields}, 103(2):143--163, 1995.

\bibitem[Hue93]{Huebner1993PhD}
M.~Huebner.
\newblock {\em Parameter Estimation for SPDEs}.
\newblock PhD thesis, University of Southern California, Los Angeles, USA,
  1993.

\bibitem[Hue97]{Huebner1997}
M.~Huebner.
\newblock A characterization of asymptotic behaviour of maximum likelihood
  estimators for stochastic {PDE}'s.
\newblock {\em Math. Methods Statist.}, 6(4):395--415 (1998), 1997.

\bibitem[Hue99]{Huebner1999}
M.~Huebner.
\newblock Asymptotic properties of the maximum likelihood estimator for
  stochastic {PDE}s disturbed by small noise.
\newblock {\em Stat. Inference Stoch. Process.}, 2(1):57--68 (2000), 1999.

\bibitem[IK98]{IbragimovKhasminskii1998}
I.~A. Ibragimov and R.~Z. Khasminskii.
\newblock Problems of estimating the coefficients of stochastic partial
  differential equations. {I}.
\newblock {\em Teor. Veroyatnost. i Primenen.}, 43(3):417--438, 1998.

\bibitem[IK99]{IbragimovKhasminskii1999}
I.~A. Ibragimov and R.~Z. Khasminskii.
\newblock Problems of estimating the coefficients of stochastic partial
  differential equations. {II}.
\newblock {\em Teor. Veroyatnost. i Primenen.}, 44(3):526--554, 1999.

\bibitem[IK00]{IbragimovKhasminskii2000}
I.~A. Ibragimov and R.~Z. Khasminskii.
\newblock Problems of estimating the coefficients of stochastic partial
  differential equations. {III}.
\newblock {\em Teor. Veroyatnost. i Primenen.}, 45(2):209--235, 2000.

\bibitem[KL85]{KoskiLoges1985}
T.~Koski and W.~Loges.
\newblock Asymptotic statistical inference for a stochastic heat flow problem.
\newblock {\em Statist. Probab. Lett.}, 3(4):185--189, 1985.

\bibitem[KLB02]{KleptsynaBreton2002}
M.~L. Kleptsyna and A.~Le~Breton.
\newblock Statistical analysis of the fractional {O}rnstein-{U}hlenbeck type
  process.
\newblock {\em Stat. Inference Stoch. Process.}, 5(3):229--248, 2002.

\bibitem[KLBR00]{KleptsynaLeBretonRoubaud2000}
M.~L. Kleptsyna, A.~Le~Breton, and M.-C. Roubaud.
\newblock Parameter estimation and optimal filtering for fractional type
  stochastic systems.
\newblock {\em Stat. Inference Stoch. Process.}, 3(1-2):173--182, 2000.
\newblock 19th ``Rencontres Franco-Belges de Statisticiens'' (Marseille, 1998).

\bibitem[Koz77]{Kozlov1977}
S.~M. Kozlov.
\newblock Equivalence of measures in {I}t\^o's linear partial differential
  equations.
\newblock {\em Vestnik Moskov. Univ. Ser. I Mat. Meh.}, (4):47--52, 1977.

\bibitem[Koz78]{Kozlov1978}
S.~M. Kozlov.
\newblock Some questions of stochastic partial differential equations.
\newblock {\em Trudy Sem. Petrovsk.}, (4):147--172, 1978.

\bibitem[Kut91]{Kutoyants1991}
Y.~A. Kutoyants.
\newblock Minimum-distance parameter estimation for diffusion-type
  observations.
\newblock {\em C. R. Acad. Sci. Paris S\'er. I Math.}, 312(8):637--642, 1991.

\bibitem[Kut04]{KutoyantsBook2004}
Y.~A. Kutoyants.
\newblock {\em Statistical inference for ergodic diffusion processes}.
\newblock Springer Series in Statistics. Springer-Verlag London Ltd., London,
  2004.

\bibitem[Log84]{Loges1984}
W.~Loges.
\newblock Girsanov's theorem in {H}ilbert space and an application to the
  statistics of {H}ilbert space-valued stochastic differential equations.
\newblock {\em Stochastic Process. Appl.}, 17(2):243--263, 1984.

\bibitem[Lot96]{Lototskiy1996PhD}
S.~V. Lototsky.
\newblock {\em Problems in statistics of stochastic differential equations}.
\newblock PhD thesis, University of Southern California, Los Angeles, USA,
  1996.

\bibitem[Lot03]{Lototsky2003}
S.~V. Lototsky.
\newblock Parameter estimation for stochastic parabolic equations: asymptotic
  properties of a two-dimensional projection-based estimator.
\newblock {\em Stat. Inference Stoch. Process.}, 6(1):65--87, 2003.

\bibitem[Lot04]{Lototsky2004}
S.~V. Lototsky.
\newblock Optimal filtering of stochastic parabolic equations.
\newblock In {\em Recent developments in stochastic analysis and related
  topics}, pages 330--353. World Sci. Publ., Hackensack, NJ, 2004.

\bibitem[Lot09]{Lototsky2009Survey}
S.~V. Lototsky.
\newblock Statistical inference for stochastic parabolic equations: a spectral
  approach.
\newblock {\em Publ. Mat.}, 53(1):3--45, 2009.

\bibitem[LR99]{LototskyRozovskii1999}
S.~V. Lototsky and B.~L. Rozovskii.
\newblock Spectral asymptotics of some functionals arising in statistical
  inference for {SPDE}s.
\newblock {\em Stochastic Process. Appl.}, 79(1):69--94, 1999.

\bibitem[LR00]{LototskyRozovskii2000}
S.~V. Lototsky and B.~L Rozovskii.
\newblock Parameter estimation for stochastic evolution equations with
  non-commuting operators.
\newblock In {\em in Skorohod's Ideas in Probability Theory, V.Korolyuk,
  N.Portenko and H.Syta (editors)}, pages 271--280. Institute of Mathematics of
  National Academy of Sciences of Ukraine, Kiev, Ukraine, 2000.

\bibitem[LR17]{LototskyRozovsky2017Book}
S.~V. Lototsky and B.~L. Rozovsky.
\newblock {\em Stochastic partial differential equations}.
\newblock Universitext. Springer, Cham, 2017.

\bibitem[LS78]{LiptserShiryayevBook1978}
R.~S. Liptser and A.~N. Shiryayev.
\newblock {\em Statistics of random processes. {II}}.
\newblock Springer-Verlag, 1978.

\bibitem[LS89]{LiptserShiryayevBookMartingales}
R.~S. Liptser and A.~N. Shiryayev.
\newblock {\em Theory of martingales}, volume~49 of {\em Mathematics and its
  Applications (Soviet Series)}.
\newblock Kluwer Academic Publishers Group, 1989.

\bibitem[LS00]{LiptserShiryayev}
R.~S. Liptser and A.~N. Shiryayev.
\newblock {\em Statistics of random processes {I}. {G}eneral theory}.
\newblock Springer-Verlag, New York, 2nd edition, 2000.

\bibitem[Mar03]{Markussen2003}
B.~Markussen.
\newblock Likelihood inference for a discretely observed stochastic partial
  differential equation.
\newblock {\em Bernoulli}, 9(5):745--762, 2003.

\bibitem[MR94]{MikuleviciusRozovskii1994}
R.~Mikulevicius and B.~L. Rozovskii.
\newblock Uniqueness and absolute continuity of weak solutions for parabolic
  {SPDE}s.
\newblock {\em Acta Appl. Math.}, 35(1-2):179--192, 1994.

\bibitem[MR02]{MikuleviciusRozovskii2002}
R.~Mikulevicius and B.~L. Rozovskii.
\newblock On martingale problem solutions for stochastic {N}avier-{S}tokes
  equation.
\newblock In {\em Stochastic partial differential equations and applications
  ({T}rento, 2002)}, volume 227 of {\em Lecture Notes in Pure and Appl. Math.},
  pages 405--415. Dekker, New York, 2002.

\bibitem[PR97]{PiterbargRozovskii1997}
L.~I. Piterbarg and B.~L. Rozovskii.
\newblock On asymptotic problems of parameter estimation in stochastic {PDE}'s:
  discrete time sampling.
\newblock {\em Math. Methods Statist.}, 6(2):200--223, 1997.

\bibitem[PR00]{PrakasaRao2000}
B.~L.~S. Prakasa~Rao.
\newblock Bayes estimation for some stochastic partial differential equations.
\newblock {\em J. Statist. Plann. Inference}, 91(2):511--524, 2000.
\newblock Prague Workshop on Perspectives in Modern Statistical Inference:
  Parametrics, Semi-parametrics, Non-parametrics (1998).

\bibitem[PR02]{PrakasaRao2002}
B.~L.~S. Prakasa~Rao.
\newblock Nonparametric inference for a class of stochastic partial
  differential equations based on discrete observations.
\newblock {\em Sankhy\=a Ser. A}, 64(1):1--15, 2002.

\bibitem[PR03]{PrakasaRao2003}
B.~L.~S. Prakasa~Rao.
\newblock Estimation for some stochastic partial differential equations based
  on discrete observations. {II}.
\newblock {\em Calcutta Statist. Assoc. Bull.}, 54(215-216):129--141, 2003.

\bibitem[PR04]{PrakasaRao2004}
B.~L.~S. Prakasa~Rao.
\newblock Parameter estimation for some stochastic partial differential
  equations driven by infinite dimensional fractional {B}rownian motion.
\newblock {\em Theory Stoch. Process.}, 10(3-4):116--125, 2004.

\bibitem[Pvs05]{Pospivsil2005PhD}
J.~Posp\'\i~\v sil.
\newblock {\em On parameter estimates in stochastic evolution equations driven
  by fractional Brownian motion}.
\newblock PhD thesis, University of West Bohemia, Plzen, 2005.

\bibitem[PvsT07]{PospivsilTribe2007}
J.~Posp\'\i~\v sil and R.~Tribe.
\newblock Parameter estimates and exact variations for stochastic heat
  equations driven by space-time white noise.
\newblock {\em Stoch. Anal. Appl.}, 25(3):593--611, 2007.

\bibitem[Roz90]{RozovskiiBook}
B.~L. Rozovskii.
\newblock {\em Stochastic evolution systems}, volume~35 of {\em Mathematics and
  its Applications (Soviet Series)}.
\newblock Kluwer Academic Publishers Group, Dordrecht, 1990.
\newblock Linear theory and applications to nonlinear filtering.

\bibitem[Shi96]{ShiryaevBookProbability}
A.~N. Shiryaev.
\newblock {\em Probability}, volume~95 of {\em Graduate Texts in Mathematics}.
\newblock Springer-Verlag, New York, second edition, 1996.

\end{thebibliography}
{\footnotesize
\def\cprime{$'$}

}
\end{document}